\documentclass[final,3p,times]{elsarticle}

\usepackage[utf8]{inputenc}
\usepackage{siunitx}
\usepackage{amsmath}
\usepackage{array}
\usepackage{amsthm}
\usepackage{amsfonts}
\usepackage{graphicx}
\usepackage{float}
\usepackage{bm}
\usepackage[final]{microtype}
\usepackage{mathtools}
\mathtoolsset{showonlyrefs}

\usepackage{todonotes}

\usepackage{hyperref}
\hypersetup{
	hidelinks,
    colorlinks=black,
    linkcolor=black,
    filecolor=black,      
    urlcolor=black,
}

\newtheorem{theorem}{Theorem}

\newtheorem{definition}{Definition}
\newtheorem{example}{Example}
\newtheorem{lemma}{Lemma}

\newtheorem{proposition}{Proposition}
\newtheorem{remark}{Remark}
\newtheorem{problem}{Problem}

\newcommand{\abs}[1]{\left\lvert#1\right\rvert}

\newcommand{\sign}[1]{\mbox{sign}(#1)}
\newcommand{\sgn}[1]{\lfloor#1\rceil}
\newcommand{\real}[1]{\mbox{Re}(#1)}
\renewcommand{\exp}[1]{\mbox{exp}(#1)}

\def\Red#1{\textcolor{red}{#1}}

\journal{ArXiv}

\begin{document}

\begin{frontmatter}

\title{A redesign methodology generating predefined-time differentiators with bounded time-varying gains
\footnote{\Red{This is the preprint version of the accepted manuscript: Aldana-López R., Seeber R., Gómez-Gutiérrez D., Angulo M.T., Defoort M. ``A redesign methodology generating predefined-time differentiators with bounded time-varying gains". Int J
Robust Nonlinear Control. 2022; 1–16. DOI: 10.1002/rnc.6315. 
\textbf{Please cite the publisher's version}. For the publisher's version and full citation details see:
\url{https://doi.org/10.1002/rnc.6315}. This article may be used for non-commercial purposes in accordance with Wiley Terms and Conditions for Use of Self-Archived Versions. 
}}
}

\author[1]{Rodrigo Aldana-López}
\author[2]{Richard Seeber}
\author[3,4]{David Gómez-Gutiérrez}
\author[5]{Marco Tulio Angulo}
\author[6]{Michael Defoort}

\address[1]{Universidad de Zaragoza, Departamento de Informatica e Ingenieria de Sistemas (DIIS), María de Luna, s/n, 50018, Zaragoza, Spain.}

\address[2]{Graz University of Technology, Christian Doppler Laboratory for Model Based Control of Complex Test Bed Systems, Institute of Automation and Control, Graz, Austria.}

\address[3]{Intel Tecnología de M\'exico, Intel Labs, Av. del Bosque 1001, Colonia El Bajío, Zapopan, 45019, Jalisco, Mexico.}

\address[4]{Instituto Tecnológico José Mario Molina Pasquel y Henríquez, Unidad Académica Zapopan, Cam. Arenero 1101, Zapopan,  45019, Jalisco, Mexico.}

\address[5]{Universidad Nacional Autonoma de México, CONACyT - Institute of Mathematics, Boulevard Juriquilla 3001, Juriquilla, Queretaro, Mexico.}

\address[6]{Polytechnic University of Hauts-de-France, LAMIH, UMR CNRS 8201, INSA, Le Mont Houy, F - 59313, Valenciennes CEDEX 9, Valenciennes, France.}

\begin{abstract}
There is an increasing interest in designing differentiators, which converge exactly before a prespecified time regardless of the initial conditions, i.e., which are fixed-time convergent with a predefined Upper Bound of their Settling Time (\textit{UBST}), due to their ability to solve estimation and control problems with time constraints. However, for the class of signals with a known bound of their $(n+1)$-th time derivative, the existing design methodologies yield a very conservative \textit{UBST}, or result in gains that tend to infinity at the convergence time. Here, we introduce a new methodology based on time-varying gains to design arbitrary-order exact differentiators with a predefined \textit{UBST}. This \textit{UBST} is a priori set as one parameter of the algorithm. Our approach guarantees that the \textit{UBST} can be set arbitrarily tight, and we also provide sufficient conditions to obtain exact convergence while maintaining bounded time-varying gains. Additionally, we provide necessary and sufficient conditions such that our approach yields error dynamics with a uniformly Lyapunov stable equilibrium. Our results show how time-varying gains offer a general and flexible methodology to design algorithms with a predefined \textit{UBST}.
\end{abstract}


\end{frontmatter}

\section{Introduction}

Real-time exact differentiators are instrumental algorithms for solving various estimation and control problems~\cite{Rios2018ASystems,Imine2011ObservationHOSM-observers,FerreiraDeLoza2015OutputIdentification,Navarro2012}. Recently, there is interest on extending their applicability to problems with time constraints~\cite{Gomez2015,Defoort2011}. In this paper, the goal is to design exact differentiators with uniform convergence despite the magnitude of the initial differentiation error, i.e., with fixed-time convergence, where the  \emph{Upper Bound for their Settling Time} (\textit{UBST}) can be set a priori by the user.
Although some methodologies have been proposed to construct exact differentiators with a predefined \textit{UBST}, several fundamental gaps remain.

On the one hand, state observer methodologies have been proposed in~\cite{Holloway2019,Menard2017}, which can be applied to the differentiation problem of signals whose $(n+1)$-th time derivative is precisely zero. The approach in~\cite{Holloway2019} is based on time-varying gains and converges precisely at the desired time. However, the resulting differentiator's gain diverges to infinity as the time approaches the user-defined \textit{UBST}. Such an unbounded gain is problematic under measurement noise or limited numerical precision. The approach in~\cite{Menard2017} is an autonomous differentiator based on homogeneity~\cite{Andrieu2008,Andrieu2009HomogeneityDesign}. However, its \textit{UBST} is greatly overestimated~\cite{Menard2017}. Recent works have been proposed to maintain the \textit{TVG} finite at the convergence time~\cite{R.Aldana-Lopez2021AGains,Orlov2022Prescribed-TimeGains}. However, the magnitude of the \textit{TVG} is still an unbounded function of the initial condition. An important property to be studied when working with time-varying systems is the uniform Lyapunov stability~\cite{Khalil2002NonlinearSystems}, as the absence of uniform stability has an inherent lack of robustness implications. However, such analysis is missing in the literature of prescribed-time observers and differentiators based on \textit{TVG}. 

On the other hand, for the broader class of signals whose $(n+1)$-th time derivative is bounded by a known constant, several differentiators have been proposed incorporating discontinuities to guarantee their exactness~\cite{Cruz-Zavala2011,Seeber2021Robust,Moreno2021arbitrary}. However, the results in~\cite{Cruz-Zavala2011,Seeber2021Robust} are limited to first-order derivatives. Additionally, the Lyapunov techniques used for their design in~\cite{Cruz-Zavala2011,Moreno2021arbitrary} often lead to very conservative estimates of the \textit{UBST} (e.g. a $130$-fold overestimate in \cite{Cruz-Zavala2011}). A conservative estimation of the \textit{UBST} results in an over-engineered predefined-time differentiator with a significant slack between the desired \textit{UBST} and the least one, which typically leads to larger than necessary differentiation errors.
Currently, there is no methodology to reduce such over-engineering in high-order differentiators.

In this paper, we fill the gaps above by introducing a novel methodology based on a class of time-varying gains to design arbitrary-order exact differentiators with fixed-time convergence and a predefined \textit{UBST} for the class of signals with a known bound on their $(n+1)$-th derivative.
Specifically, our methodology is based on a class of Time-Varying Gains (\textit{TVG}) that subsumes the one used in~\cite{Holloway2019}. However, we derive sufficient conditions to guarantee that the \textit{TVG} of the differentiator remain bounded.
This is in contrast to workarounds suggested in \cite{Holloway2019}, where the \textit{TVG} is maintained bounded at the cost of losing the exactness of the differentiator, producing errors at the prescribed time that are an unbounded function of the initial condition.
Furthermore, we prove that our methodology enables us to set the actual worst-case convergence time of the differentiator arbitrarily close to the desired \textit{UBST}. Since, the resulting differentiator is time-varying, we provide necessary and sufficient conditions for our methodology to yield an error dynamic with an uniformly Lyapunov stable equilibrium.

The rest of this paper is organized as follows. In Section~\ref{Sec:ProbStatement}, we introduce the predefined-time differentiation problem and present our Main Result. Section~\ref{Sec:Sim} presents numerical examples illustrating our contributions, in particular, comparing our differentiators with state-of-the-art algorithms. Finally,  Section~\ref{Sec:Conclu}  presents some concluding remarks. 
The proofs are collected in the Appendix.

\vspace{0.3cm}

\noindent \textbf{Notation:}

Let $\mathbb{R}_+=\{x\in\mathbb{R}\,:\,x\geq0\}$ and $\Bar{\mathbb{R}}_+=\mathbb{R}_+\cup\{\infty\}$. For $x\in\mathbb{R}$, $\sgn{x}^\alpha = |x|^\alpha \mbox{sign}(x)$, if $\alpha\neq0$ and $\sgn{x}^\alpha = \mbox{sign}(x)$ if $\alpha=0$. 
For a function $\phi:\mathcal{I}\to\mathcal{J}$, its reciprocal $\phi(\tau)^{-1}$, $\tau\in\mathcal{I}$,  is such that $\phi(\tau)^{-1}\phi(\tau)=1$ and its inverse function $\phi^{-1}(t)$, $t\in\mathcal{J}$, is such that $\phi(\phi^{-1}(t))=t$. For functions $\phi,\psi:\mathbb{R}\to\mathbb{R}$, $\phi\circ\psi(t)$ denotes the composition $\phi(\psi(t))$. We use boldface lower case letter for vector and boldface capital letters for matrices. Given a matrix $\mathbf{A}\in\mathbb{R}^{n\times m}$, $\mathbf{A}^T$ represents its transpose. Given a vector $\mathbf{v}\in\mathbb{R}^{n}$, $\|\mathbf{v}\|=\sqrt{\mathbf{v}
^T\mathbf{v}}$. We use the notation $\mathbf{b}_{i}\in\mathbb{R}^{(n+1)\times 1}$, $i\in\{0,\dots,n\}$, to denote a vector filled with zeros, except for the $(i+1)$-th component which is $1$. For a signal $y:\mathbb{R}_+\to\mathbb{R}$, $y^{(i)}(t)$ represents its $i-$th derivative with respect to time at time $t$. To denote a first-order derivative of $y(t)$, we simple use the notation $\dot{y}(t)$. The notation $\real{\lambda}$ denotes the real part of the complex number $\lambda$.

\section{Problem statement and Main Result}
\label{Sec:ProbStatement}

\subsection{Problem statement and preliminaries}
We consider time-varying differentiation algorithms written as the dynamical system
\begin{align}
\dot{z}_{i}&=-h_i(e_0,t;p)+z_{i+1}, \quad  i=0,\ldots,n-1, \\
\dot{z}_n&=-h_n(e_0,t;p), \label{Eq:Diff}
\end{align}
where $n>0$, $e_0(t)=z_0(t)-y(t)$, for some scalar signal $y(t)$. Above, $p$ is used to highlight some parameters of interest and $\{h_i\}_{i=0}^n$ are the \emph{correction functions} of the algorithm, which could be discontinuous in the first argument\footnote{In the spirit of Filippov's interpretation of differential equations, solutions of~\eqref{Eq:Diff} are understood as any absolutely continuous function that satisfies the differential inclusion obtained by applying the Filippov regularization to the right-hand side of~\eqref{Eq:Diff} (See \cite[Page 85]{Filippov1988DifferentialSides}), allowing us to consider a discontinuous in the first argument right-hand side of~\eqref{Eq:Diff}. In the usual Filippov's interpretation of the solutions of $\dot{x}=f(\mathbf{x},t)$, $\mathbf{x}\in\mathbb{R}^{n}$ it is assumed that $\|\mathbf{f}(\mathbf{x}, t)\|$ has an integrable majorant function of time for any $\mathbf{x}$, ensuring existence and uniqueness of solutions in forward time. However, in this work we deal with $\mathbf{f}(\mathbf{x}, t)$ for which no majorant function exist, but existence and uniqueness of solutions is still guaranteed by argument similar to \cite{aldana2019design}. In particular, existence of solutions follows directly from the equivalence of solutions to a well-posed Filippov system via the time-scale transformation.}.

Defining  $e_i(t)=z_i(t)-y^{(i)}(t)$, $i=0, \ldots, n$, the differentiation error dynamics %
is 
\begin{align}
\dot{e}_{i}&=-h_i(e_0,t;p)+e_{i+1}, \quad  i=0,\ldots,n-1, \\
\dot{e}_n&=-h_n(e_0,t;p)+d(t), \label{Eq:DiffErr}
\end{align}
where $d(t)=-y^{(n+1)}(t)$.
We let  $\mathbf{e}(t) := [e_0(t),\ldots,e_n(t)]^T$.

We denote as $\mathcal Y_{\mathcal{L}(t)}^{(n+1)}$ the class of all scalar signals $y(t)$ defined for $t \ge 0$ such that $\left|y^{(n+1)}(t)\right|\leq \mathcal{L}(t)$ for all $t \geq 0$, for some known function $\mathcal{L}(t) \geq 0$. When $\mathcal{L}(t)=L$ is constant, we simply write $\mathcal Y_{L}^{(n+1)}$. When  the correction functions $\{h_i\}_{i=0}^n$ are such that the origin of Eq. \eqref{Eq:DiffErr} is globally asymptotically stable~\cite{Khalil2002NonlinearSystems} for scalar signals $y$ of class $\mathcal{Y}^{(n+1)}_{\mathcal{L}(t)}$ and some $\mathcal{L}(t)$, then its \emph{settling time function} is
$$
T(\mathbf{e}(0))=\inf\left\{\xi \in \Bar{\mathbb{R}}_+: \lim_{\tau\to\xi}\sup_{t \ge \tau} \|\mathbf{e}\left(t;\mathbf{e}(0),y\right)\|=0 \quad \forall y \in \mathcal Y_{\mathcal{L}(t)}^{(n+1)} \right\},$$
where $\mathbf{e}(0)$  is the initial differentiation error. 
Here, $\mathbf{e}\left(t;\mathbf{e}(0),y\right)$ is the solution of Eq.~\eqref{Eq:DiffErr} starting at $\mathbf{e}(0)$ with signal $y\in \mathcal Y_{\mathcal{L}(t)}^{(n+1)}$. With some abuse of notation, we write $\mathbf{e}(t)=\mathbf{e}\left(t;\mathbf{e}(0),y\right)$ when there is no ambiguity.
Then, the origin of system~\eqref{Eq:DiffErr} is globally \textit{finite-time stable} if it is globally asymptotically stable and $T(\mathbf{e}(0))<\infty$. 
The origin of system~\eqref{Eq:DiffErr} is globally \textit{fixed-time stable} if it is globally finite-time stable and there exists $T_{{\max}}<\infty$ such that $T(\mathbf{e}(0))\leq T_{\max}$ for all $\mathbf{e}(0)\in\mathbb{R}^{n+1}$. Here,  $T_{\max}$ is an \emph{Upper Bound of the Settling Time} (\textit{UBST})~\cite{Sanchez-Torres2018,Jimenez2019}. We say that $T^*_{\max}$ is the least \textit{UBST}~\cite{Aldana-Lopez2018} of system~\eqref{Eq:DiffErr} if $$T^*_{\max}:=\sup_{\mathbf{e}(0)\in \mathbb{R}^{n+1}} T(\mathbf{e}(0)).$$
Finally, the origin of system~\eqref{Eq:DiffErr} is \emph{uniformly Lyapunov stable} if for every $\epsilon > 0$ there exists $\delta > 0$ such that for all $s \ge 0$, $||\mathbf{e}(s)|| \le \delta$ implies $||\mathbf{e}(t)|| \le \epsilon$ for all $t \ge s$.

With the above definitions, a differentiator is said to be \emph{asymptotic}, \emph{exact} or \emph{fixed-time} if the origin of its error dynamic is globally asymptotically stable, globally finite-time, or globally fixed-time stable, respectively. Moreover, a differentiator is said to be \emph{predefined-time} if it is fixed-time with a desired (predefined) \textit{UBST}. Additionally, we say that the differentiator is time-invariant if the correction functions $\{h_i\}_{i=0}^n$ are independent of $t$.

Our problem is:

\begin{problem}[Predefined-time $n$-th order exact differentiation]
\label{Problem} Given a desired convergence time $T_c > 0$ and any signal $y \in \mathcal Y_L^{(n+1)}$, $L\geq0$, obtain estimates $z_i(t)$ of the time derivatives $y^{(i)}(t)$, $i = 0, \cdots, n$, such that the identities $z_i(t) \equiv y^{(i)}(t)$  hold for all $t \geq T_c$. 
\end{problem}

To solve Problem~\ref{Problem}, below we provide a methodology to design the correction functions $\{h_i\}_{i=0}^n$ for the differentiator algorithm in Eq.~\eqref{Eq:Diff}, to obtain a predefined-time differentiator. To highligh that the differentiator is designed using the desired \textit{UBST} $T_c$ as a parameter, we write $h_i(e_0,t;T_c)$, $i=0,\ldots,n$.

\subsection{Main results}
\label{Sec:MainRes}

Let the class of functions $\mathcal Y_L^{(n+1)}$, $L\geq0$ and the desired convergence time $T_c > 0$ be given. Our main result is a method to ``redesign'' an existing time-invariant asymptotic differentiator to obtain a time-varying predefined-time differentiator with an \textit{UBST} given by $T_c$. We start with an asymptotic differentiator having the form:
\begin{align}
    \dot{z}_i&=-\phi_i(e_0)+z_{i+1}, \quad  i=0,\ldots,n-1,\\
    \dot{z}_n&=-\phi_n(e_0). \label{Eq:BaseDiff}
\end{align}
Here,  $\{\phi_i \}_{i=0}^n$ are its specific correction functions. We call algorithm~\eqref{Eq:BaseDiff} a base differentiator and $\{\phi_i \}_{i=0}^n$ \textit{admissible correction functions} if $\{\phi_i\}_{i=0}^n$ are continuous, except perhaps at the origin, and there exist an interval $\mathcal{I}_\phi\subseteq\mathbb{R}_+$ such that for any $\alpha\in \mathcal{I}_\phi$:
\begin{enumerate}
    \item[(A1)] \label{it:admissibility} the algorithm~\eqref{Eq:BaseDiff} is an asymptotic differentiator for the class $\mathcal{Y}_{\mathcal{L}(t)}^{(n+1)}$, where $$\mathcal{L}(t)=L\exp{-\alpha(n+1) t},$$ and there exists $\gamma(\mathbf{e}(0),\alpha)>0$ such that differentiation error vector $\mathbf{e}(t)$ in algorithm~\eqref{Eq:BaseDiff} for any $y\in\mathcal{Y}_{\mathcal{L}(t)}^{(n+1)}$ has an exponential convergence of the form:
    $$
    \|\mathbf{e}(t)\| < \gamma(\mathbf{e}(0),\alpha) \exp{-\alpha (n+1)t}
    $$
\end{enumerate}

Next, let $T_f \in \Bar{\mathbb{R}}_+$ denote a bound for the settling time function of the base differentiator for the class of signals $\mathcal{Y}_{\mathcal{L}(t)}^{(n+1)}$. If such bound does not exist or is unknown, we simply set $T_f = \infty$.

Lemma~\ref{Lem:Aux}, in the Appendix, shows that the correction functions as given in Table~\ref{Tab:phi} are admissible correction functions of asymptotic differentiators, and provides bounds for their convergence time. See also \cite{Andrieu2009HomogeneityDesign,Levant2003,Cruz-Zavala2011,Levant2019,Castillo2018Super-TwistingPerturbations} for additional examples.

\begin{table*}
    \centering
    \begin{tabular}{|c|p{5.6cm}|p{3.5cm}|p{5.5cm}|}
    \hline
     & {\bf Correction functions} & $\mathcal{I}_\phi$, $T_f$ & {\bf Design conditions} \\
     \hline
    \hline
    \textit{i)} &
    $\begin{array}{ll}\phi_i(w)=r^{i+1}l_iw \end{array}$ &  $\begin{array}{l}
    \mathcal{I}_\phi=\left[0,r\right)\\
    T_f=\infty
    \end{array}$ & 
    $s^{n+1}+l_0s^n+\cdots+l_{n}$ is a Hurwitz polynomial with roots $\lambda_i\in\mathbb{C}$ and $\max(\real{\lambda_0},\ldots,\real{\lambda_{n}})=-(n+1)$. $r>0$.
   \\
    \hline
    \textit{ii)} & $\begin{array}{ll}\phi_i(w)=l_iL^{\frac{i+1}{n+1}}\sgn{w}^{\frac{n-i}{n+1}}\end{array}$ & $\mathcal{I}_\phi=\mathbb{R}_+$, $T_f=\infty$  & $\{l_i\}_{i=0}^n$ are chosen as in~\cite{Levant2019,Cruz2018,Sanchez2016construction}. For instance, for $n=2$, $l_0=2.0$, $l_1=2.12$ and $l_2=1.1$~\cite{Levant2019}.\\
    \hline
    \textit{iii)} & $\begin{array}{l}\phi_0(w)=4\sqrt{L}(\sgn{w}^{\frac{1}{2}}+k\sgn{w}^{\frac{3}{2}})\\ \phi_1(w)=2L(\sign{w}+4k^2w+3k^4\sgn{w}^2)\end{array}$ &  $\mathcal{I}_\phi=\mathbb{R}_+$,  $T_f=T^*_{\max}$ &  Only if $n=1$. $k=\frac{9.8}{L^{\frac{1}{2}}T_{\max}^*}$ as in \cite{Seeber2020ExactBound}. \\
    \hline
    \textit{iv)} & $\begin{array}{ll}\phi_i(w)=\theta^{i+1}k_i\left(
    \sgn{w}^{(i+1)c-i}\right.+\left.\sgn{w}^{(i+1)b+i}\right)\end{array}$ & $\mathcal{I}_\phi=\mathbb{R}_+$, $T_f=\frac{4}{\theta}((1-c)^{-1} + (b-1)^{-1})$ & Only if $L=0$. $\theta\geq 1$,  $c\in(1-\epsilon,1), b\in(1,1+\epsilon)$, $\epsilon>0$ sufficiently small and $\{k_i\}_{i=0}^n$ chosen as in \cite{Menard2017}. \\
    \hline
    \end{tabular}
    \caption{Examples of admissible correction functions $\{\phi_i\}_{i=0}^n$ for the base differentiator algorithm \eqref{Eq:BaseDiff}.}
    \label{Tab:phi}
\end{table*} 

To introduce our main result, we make the following definitions based on the correction functions $\{\phi_i \}_{i=0}^n$ of the base differentiator. First, let $\alpha\in\mathcal{I}_\phi$ and define
$$
\mathbf{Q}: = \begin{bmatrix}
(\mathbf{U}-\alpha \mathbf{D})^{n} \mathbf{b}_{n}; &
\cdots; &
(\mathbf{U}-\alpha \mathbf{D}) \mathbf{b}_{n}; &
\mathbf{b}_{n}
\end{bmatrix},
$$
where $\mathbf{b}_{n}:=[0,\cdots,0,1]^T\in\mathbb{R}^{(n+1)\times 1}$, $\mathbf{D}:=\text{diag}\{0,\ldots,n\}$ and
$\mathbf{U}:=[u_{ij}]\in\mathbb{R}^{(n+ 1)\times (n+ 1)}$, with $u_{ij}=1$ if $j=i+1$ and $u_{ij}=0$ otherwise, i.e.,
\begin{equation*}
    \mathbf{b}_n = \begin{bmatrix}
    0 \\ \vdots \\ 0 \\ 1
    \end{bmatrix} \qquad
    \mathbf{D} = \begin{bmatrix}
    0 & 0 & \cdots & 0  \\
    0 & 1 & \cdots & 0  \\
    \vdots & \vdots & \ddots & \vdots \\
    0 & 0 & \cdots & n
    \end{bmatrix}
    \qquad \mathbf{U} = \begin{bmatrix}
        0 & 1 & \cdots & 0 \\
        \vdots & \ddots & \ddots & \vdots \\
        0 & \ddots & \ddots & 1 \\
        0 & 0 & \cdots & 0
    \end{bmatrix}.
\end{equation*}
Second, let $\Phi(e_0):=[\phi_0(e_0),\cdots,\phi_n(e_0)]^T$.
Third, define the function $\mathbf{f}: \mathbb R\rightarrow \mathbb R^{n+1}$, as 
$$
\mathbf{f}(e_0):=\beta\mathbf{Q}\Phi(\beta^{-1}e_0)+(\mathbf{U}-\alpha \mathbf{D})^{n+1} \mathbf{b}_{n}e_0.
$$ 
Here, $\beta\geq(\alpha T_c/\eta)^{n+1}$, with $\eta:=1-\exp{-\alpha T_f}$. %
 Finally, for $i = 0, \ldots, n$, define $g_i(e_0; L) :=
       l_i \max(L,\mu)^{\frac{i+1}{n+1}}\sgn{e_0}^{\frac{n-i}{n+1}} $
 where $l_i \in \mathbb{R}_+$ is chosen as in Levant's  arbitrary-order exact differentiator~\cite{Levant2003,Cruz2018,Sanchez2016construction} and $\mu>0$ is an arbitrary small positive constant.

Using the above notation, our main result is the following redesigned differentiator:
\begin{theorem}
\label{Th:MainResult} Given the class of functions $\mathcal Y_L^{(n+1)}$, $n>0$, $L \ge 0$, and admissible correction functions $\{\phi_i \}_{i=0}^n$. For a given non-negative $T_c$, a positive design parameter $\alpha\in\mathcal{I}_\phi$, and any $\mu > 0$, let functions $f(e_0)$ and $\{g_i\}_{i=0}^n$ be defined as above. Consider the following ``redesigned'' correction functions:
\begin{equation}
\label{Eq:hFunc}
    h_i(e_0,t;T_c):=\left\lbrace
    \begin{array}{ll}
        \kappa(t)^{i+1}f_i(e_0) & \text{for } t\in[0,T_c), \\
        g_i(e_0; L) & \text{otherwise,}
    \end{array}
    \right.
\end{equation}
where $f_i(e_0)$, $i=0,\ldots,n$, is the $(i+1)-$th row of $\mathbf{f}(e_0)$ and $\kappa(t)$ is a \textit{TVG} given by
\begin{equation}
\label{eq:kappa}
\kappa(t):=
\left\lbrace
\begin{array}{cc}
  \frac{\eta}{\alpha(T_c- \eta t)}  &  \text{for }t\in[0,T_c),\\
    1 & \text{otherwise}.
\end{array}
\right.
\end{equation}
Then, the differentiator of Eq. ~\eqref{Eq:Diff} is exact in $\mathcal Y_L^{(n+1)}$ and it converges before $T_c$. That is, with the above correction functions, the origin of system~\eqref{Eq:DiffErr} is fixed-time stable and $T_c$ is an \textit{UBST}. 
\end{theorem}

\begin{proof}
See Appendix~\ref{App:ProofMain}.
\end{proof}

Given that the redesigned differentiator is time-varying, we next provide necessary and sufficient conditions such that the resulting differentiator is uniformly Lyapunov stable with respect to time. 

\begin{proposition}
\label{Lemma:Uniform} Let $n \ge 1$. Under the construction in Theorem \ref{Th:MainResult}, the differentiation error in \eqref{Eq:DiffErr} is uniformly Lyapunov stable if and only if $\kappa(t)$ is uniformly bounded.
\end{proposition}
\begin{proof}
See Appendix~\ref{appendix_propositions}.
\end{proof}

\begin{remark}
Note that for $n = 0$, uniform Lyapuonv stability is obtained also for unbounded $\kappa(t)$, because the error dynamics then are a scalar system.
In the context of differentiation, this special case is less important, however, because no derivatives are computed in this case.
\end{remark}
\begin{remark}
Absence of uniformity (with respect to time) is a significant problem in practice. Suppose that briefly before $T_c$, the differentiator error is accidentally perturbed (by measurement noise, round-off errors, etc.). Proposition~\ref{Lemma:Uniform} and its proof show that, depending on how close to $T_c$ this happens, an arbitrarily large peak on the differentiator error may occur.
Below, we show how to achieve the uniformly bounded gain required for uniform Lyapunov stability for $n \ge 1$.
\end{remark}

Compared to the base differentiator of Eq. \eqref{Eq:BaseDiff}, the redesigned differentiator of Eq. \eqref{Eq:Diff}  contains two additional parameters $
\alpha,\beta$ to be tuned. These parameters allow to tune the transient response, but for any admissible value of these parameters, the predefined-time convergence is maintained. The parameter selection of Eq. \eqref{Eq:Diff} can be mostly performed using existing criteria for the base differentiator (see Table \ref{Tab:phi} for examples). When our approach is applied to the case where $T_f<\infty$, our redesign methodology subsumes the one recently suggested in~\cite{Moreno2021arbitrary}, and it provides an additional degree of freedom $\alpha$ to tune the transient behavior. Moreover, as we show below, this parameter can be used to make the differentiator converge arbitrarily close to the desired \textit{UBST} $T_c$, as shown in the following result:

\begin{proposition}
\label{Prop:TightBound}
Assume that the base differentiator is such that $T_f<\infty$ and let $T^*_{\alpha}$ denote the least upper bound for the settling time of Eq. \eqref{Eq:Diff} with the parameter $\alpha\in\mathcal{I}_\phi$ with unbounded $\mathcal{I}_\phi$. Then, for any  $\varepsilon >0$, there exists $\alpha\in\mathcal{R}_+$ such that $T_c - T_{\alpha}^* \leq \varepsilon$.
\end{proposition}
\begin{proof}
See Appendix~\ref{appendix_propositions}.
\end{proof}
In the simulation results of the following section, we show that $\alpha,\beta$  can also be used to improve the transient performance and reduce the over-engineering of the differentiator (i.e., an overly conservative bound $T_c$). 

Note that the prescribed-time observer proposed by Holloway et al. \cite{Holloway2019}, which can be seen as an exact predefined-time differentiator for signals of class $\mathcal{Y}_{0}^{(n+1)}$, has the interesting property that the settling time function~\eqref{eq:settling_time} is such that  $T(\mathbf{e}(0))=T_c$ for every nonzero initial condition. However, the time-varying gain will be such that $\lim_{t\to T_c}\kappa(t)=\infty$. These features are also present in our approach when the base differentiator is not an exact differentiator (e.g. if the base differentiator is linear as in Table~\ref{Tab:phi}-i)). %
To maintain a bounded gain, Holloway et al. suggest to ``turn off" the error correction terms at some time $t_{stop}<T_c$. However, turning-off the error correction functions in such a way will destroy the exactness of the differentiator in~\cite{Holloway2019} (the differentiation error will no longer be zero at time $T_c$). In fact, it follows from~\cite[Eq.~(27)]{Holloway2019}, that there exist constants $M,\delta>0$ and an integer $m\in\mathbb{N}$ such that
$
\|\mathbf{e}(t_{stop})\|\leq M \left(\frac{T_c-t_{stop}}{T_c}\right)^{m+1}\exp{-\delta t_{stop}}\|\mathbf{e}(0)\|
$.
Thus, the bound for $\|\mathbf{e}(t_{stop})\|$ is an unbounded function of the initial condition. 
Our methodology circumvents this limitation in two ways. First, we provide sufficient conditions such that the gain is finite at the convergence time. Nevertheless, the \textit{TVG} may grow as the magnitude of the initial error grows. Second, we provide sufficient conditions to guarantee that the gain remains bounded regardless of the initial condition. The following proposition formalizes these two points:

\begin{proposition}
\label{Prop:Bound}
Consider the base differentiator \eqref{Eq:BaseDiff} and the redesigned differentiator as in Theorem~\ref{Th:MainResult}.
\begin{itemize}
    \item[(i)] If the origin of the error dynamics of the base differentiator of Eq.~\eqref{Eq:BaseDiff} is globally finite-time stable, then in the redesigned differentiator $\kappa(T(\mathbf{e}(0)))<\infty$ for all $\mathbf{e}(0)\in\mathbb{R}^{n+1}$.
    \item[(ii)] If the origin of the error dynamics of the base differentiator of Eq.~\eqref{Eq:BaseDiff} is globally fixed-time stable, then, there exists $\kappa_{\max} <\infty$ such that in the redesigned differentiator $\kappa(T(\mathbf{e}(0))) \leq \kappa_{\max}$ for all $\mathbf{e}(0)\in\mathbb{R}^{n+1}$. In particular, if $T_f < \infty$ is a known \textit{UBST} of the base differentiator~\eqref{Eq:BaseDiff}, then 
\begin{equation}
\label{eq:gain_bound}
\kappa(t)\leq 
\kappa_{\max} :=  \frac{\exp{\alpha T_f}-1}{\alpha T_c} \ \ \text{ for all } t\geq 0.
\end{equation}

\end{itemize}
\end{proposition}
\begin{proof}
See Appendix~\ref{appendix_propositions}.
\end{proof} 

We also note that our approach yields a tradeoff for $\kappa(t)$: 
if one tries to make the convergence ``tighter'' by adjusting $\alpha$, then this necessarily yields a bigger bound $\kappa_{\max}$. Conversely,  a smaller $\kappa_{\max}$ will result in a larger ``slack'' between $T_c$ and the least \textit{UBST} of~\eqref{Eq:DiffErr}. 

\begin{remark}
Note that even if there is a bound $\kappa_{\max}$ for $\kappa(t)$, in practice, to avoid large values of $\kappa(t)$, one may wish to detect the convergence of the differentiator by monitoring $|e_0(t)|$ as in~\cite{Angulo2009OnHOSMs} and make the switching in the $\{h_i\}_{i=0}^n$ functions when the convergence is detected.
\end{remark}

\begin{remark} 
\label{Rem:Filter}
It is straightforward to extend our methodology to filtering differentiators \cite{Levant2019,Carvajal-Rubio2020}. Specifically, let $\{h_i\}_{i=0}^n$ selected as in Theorem~\ref{Th:MainResult} for signals in $\mathcal{Y}^{(n+1)}_L$. Then, the algorithm
\begin{align}
\dot{w}_i=&-h_{i-1}(w_1,t;T_c)+w_{i+1}, \intertext{ for  $i= 1,\ldots,n_f-1$,}
\dot{w}_i=&-h_{i-1}(w_1,t;T_c)+(z_0-y), \intertext{ for $i=n_f$,} \label{Eq:FilDiff} 
\dot{z}_{i-n_f-1}=&-h_{i-1}(w_1,t;T_c)+z_{i-n_f}, \intertext{ for  $i=n_f+1,\ldots,n_d+n_f$,} 
\dot{z}_{n_d}=&-h_n(w_1,t;T_c),  
\end{align}
with $n=n_d+n_f$, is a predefined-time exact differentiator but now for signals in $\mathcal{Y}^{(n_d+1)}_L$. However, in this case, we obtain that, for all $t\geq T_c$ and every initial condition $w_i(0)$, $i=1,\ldots,n_f$, $z_i(0)$, $i=0,\ldots,n_d$: $w_i(t)=0$, $i=1,\ldots,n_f$, and  $z_i(t)=y^{(i)}(t)$ for all $i=0,\ldots,n_d$. For $n_f=0$, $w_1(t)$ is defined as $w_1(t)=z_0(t)-y(t)$. This observation follows by noticing that by setting the differentiation errors $e_i(t)=w_{i+1}(t)$, for $i=0,\ldots,n_f-1$ and $e_i(t)=z_{i-n_f}(t)-y^{\left(i-n_f\right)}(t)$ for $i=n_f,\ldots,n_d+n_f$ we obtain the error dynamics~\eqref{Eq:DiffErr} where $n=n_d+n_f$ and $d(t)=-y^{(n_d+1)}(t)$ with $|d(t)|\leq L$. The filtering properties of this algorithm can be very useful in the presence of noise~\cite{Levant2019}, as we numerically illustrate in the following section.
\end{remark}

\begin{remark}
The proof of Theorem~\ref{Th:MainResult} given in Appendix~\ref{App:ProofMain} is obtained by relating, by the coordinate change~\eqref{Eq:Transfom} and the time-scale transformation given in Lemma~\ref{Lemma:ParTrans}, the differentiator's error dynamics~\eqref{Eq:DiffErr} with an asymptotically stable ``auxiliary system"~\eqref{Eq:AuxSyst} that is built with the selection of correction functions $\{\phi_i\}_{i=0}^n$.  Following the same ideas, different redesigned correction functions can also be obtained. For instance, the proof may straightforwardly be applied to a differentiator with correction functions

\begin{equation}
\label{Eq:hFuncAlternative}
    h_i(e_0,t;T_c):=\left\lbrace
    \begin{array}{ll}
        \kappa(t)^{1+i-\rho}f_{\rho,i}(e_0,t) & \text{for } t\in[0,T_c), \\
        g_i(e_0; L) & \text{otherwise,}
    \end{array}
    \right.
\end{equation}
parametrized
\footnote{The correction functions in Theorem~\ref{Th:MainResult} correspond to the case $\rho = 0$. The first-order differentiator in~\cite{Orlov2022Prescribed-TimeGains} is obtained by redesigning Levant's super-twisting algorithm~\cite{Levant1998RobustTechnique} with $\rho=n$.}
using $\rho\in[0,n+1]$, where $f_{\rho,i}(e_0,t)$ is the $(i+1)-$th component of
$$
\mathbf{f}_{\rho}(e_0,t):=\beta\mathbf{Q}_{\rho}\Phi(\beta^{-1}\kappa(t)^\rho e_0)+\kappa(t)^{\rho}(\mathbf{U}-\alpha \mathbf{D}_{\rho})^{n+1} \mathbf{b}_{n}e_0.
$$ 
with $\mathbf{D}_{\rho}:=\text{diag}\{-\rho,1-\rho,\ldots,n-\rho\}$, $
\mathbf{Q}_{\rho}: = \begin{bmatrix}
(\mathbf{U}-\alpha \mathbf{D}_{\rho})^{n} \mathbf{b}_{n}; &
\cdots; &
(\mathbf{U}-\alpha \mathbf{D}_{\rho}) \mathbf{b}_{n}; &
\mathbf{b}_{n}
\end{bmatrix},$ and $\beta\geq(\alpha T_c/\eta)^{n+1-\rho}$. The corresponding error dynamics are still related with the ``auxiliary system"~\eqref{Eq:AuxSyst}, but with 
$\pi(\tau)=\beta^{-1}(\alpha T_c/\eta)^{n+1-\rho}\exp{-\alpha(n+1-\rho) \tau}d(\varphi^{-1}(\tau))$, where $\varphi^{-1}(\tau)=t=\eta^{-1}T_c(1-\exp{-\alpha\tau})$ is the parameter transformation from Lemma \ref{Lemma:ParTrans} in Appendix \ref{Appendix:PrelimTimScale}) by the coordinate transformation
$$
\mathbf{e}(t) = \beta \mathbf{K}_{\rho}(t) \mathbf{Q}_{\rho} \mathbf{\chi}(\varphi(t)),
$$
where $\mathbf{K}_{\rho}(t):=\textrm{diag}(\kappa(t)^{-\rho},\kappa(t)^{1-\rho},\ldots,\kappa(t)^{n-\rho})$,
and the time-scale transformation given in Lemma~\ref{Lemma:ParTrans}.
\end{remark}

\section{Numerical analysis and comparisons}
\label{Sec:Sim}

Here we present case studies to analyze the performance of the redesigned differentiators, showing in particular the slack of the \textit{UBST} given by $T_c$, boundedness of the \textit{TVG}, and sensitivity to noise.  The simulations below were created in OpenModelica using the Euler integration method with a step size of $1\times 10^{-7}$.

\begin{example}
\label{Ex:FOD}
Here, we design a first-order exact fixed-time differentiator whose time-varying gain remains bounded. We also systematically study its performance against noises of different magnitudes.
Consider the base differentiator with correction functions $\phi_i(w,\tau;)$ given in Table~\ref{Tab:phi}-\textit{iv)} with $T_f=T_{\max}^*=1$. We set $\alpha=3$ and $\beta=2(\alpha T_c/\eta)^{n+1}$, and for $g_i(e_0,t)$, we choose $l_0=1.5$ and $l_1=1.1$. Thus,  
$$
\mathbf{Q}=
\left[
\begin{array}{rr}
    1   &  0 \\
    -\alpha  &   1
\end{array}
\right].
$$

In Figure~\ref{fig:FOD}a-c, we show the simulation of algorithm~\eqref{Eq:Diff}, with a desired \textit{UBST} given by $T_c=1$, to differentiate signals with different bounds of the second derivative. That is, different values of $L$ and different noise magnitude are considered. In~Figure \ref{fig:FOD}d, we show the simulation for the algorithm as a filtering differentiator~\eqref{Eq:FilDiff} with $n_f=1$ and $n_d=0$. Recall that, in the absence of noise $w_1(t)=0$ and $z_0(t)=y(t)$ for all $t\geq T_c$. Figure \ref{fig:FOD}d shows the behavior of this algorithm under noise. In all examples, the initial conditions were chosen as $z_0(0)=z_1(0)=10$.  
\begin{figure*}
    \centering
\def\svgwidth{16.0cm}
\begingroup%
  \makeatletter%
  \providecommand\color[2][]{%
    \errmessage{(Inkscape) Color is used for the text in Inkscape, but the package 'color.sty' is not loaded}%
    \renewcommand\color[2][]{}%
  }%
  \providecommand\transparent[1]{%
    \errmessage{(Inkscape) Transparency is used (non-zero) for the text in Inkscape, but the package 'transparent.sty' is not loaded}%
    \renewcommand\transparent[1]{}%
  }%
  \providecommand\rotatebox[2]{#2}%
  \newcommand*\fsize{\dimexpr\f@size pt\relax}%
  \newcommand*\lineheight[1]{\fontsize{\fsize}{#1\fsize}\selectfont}%
  \ifx\svgwidth\undefined%
    \setlength{\unitlength}{706.58581543bp}%
    \ifx\svgscale\undefined%
      \relax%
    \else%
      \setlength{\unitlength}{\unitlength * \real{\svgscale}}%
    \fi%
  \else%
    \setlength{\unitlength}{\svgwidth}%
  \fi%
  \global\let\svgwidth\undefined%
  \global\let\svgscale\undefined%
  \makeatother%
  \begin{picture}(1,0.30163763)%
    \lineheight{1}%
    \setlength\tabcolsep{0pt}%
    \put(0,0){\includegraphics[width=\unitlength,page=1]{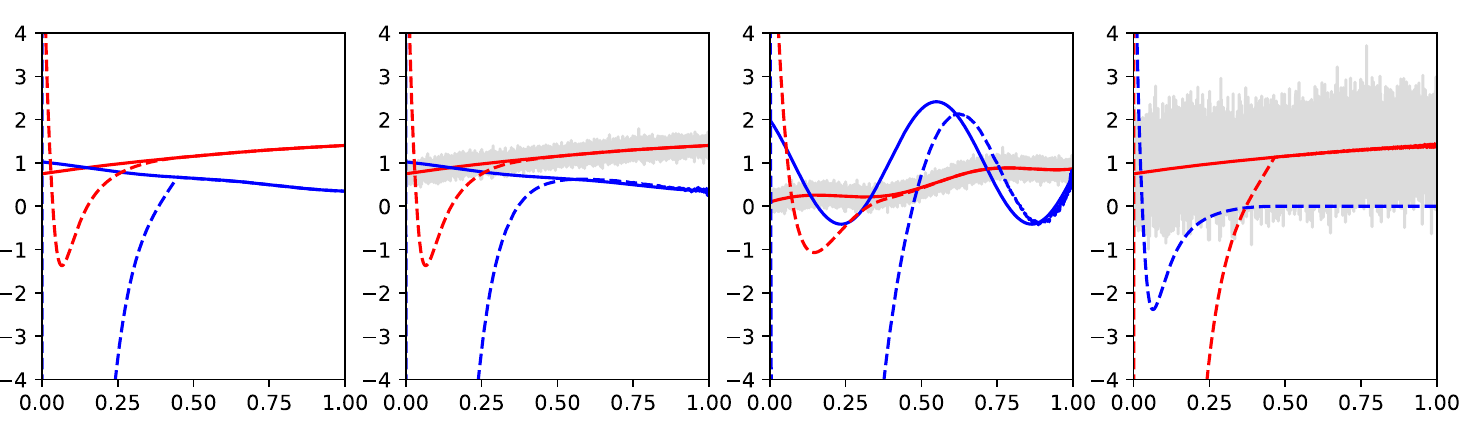}}%
    \put(0,0){\includegraphics[width=\unitlength,page=1]{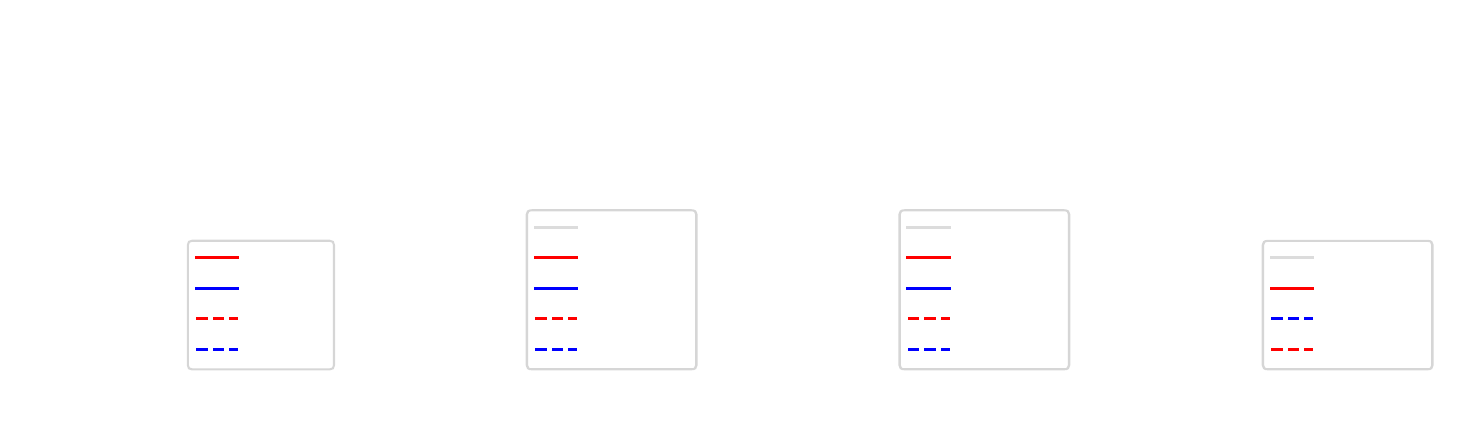}}%
    \tiny{
    \put(0.11647682,0.0002004){\color[rgb]{0,0,0}\makebox(0,0)[lt]{\lineheight{1.25}\smash{\begin{tabular}[t]{l}time\end{tabular}}}}%
    \put(0.36369134,0.0002004){\color[rgb]{0,0,0}\makebox(0,0)[lt]{\lineheight{1.25}\smash{\begin{tabular}[t]{l}time\end{tabular}}}}%
    \put(0.85812032,0.0002004){\color[rgb]{0,0,0}\makebox(0,0)[lt]{\lineheight{1.25}\smash{\begin{tabular}[t]{l}time\end{tabular}}}}%
    \put(0.61090587,0.0002004){\color[rgb]{0,0,0}\makebox(0,0)[lt]{\lineheight{1.25}\smash{\begin{tabular}[t]{l}time\end{tabular}}}}%
    \put(0.12536364,0.290885){\color[rgb]{0,0,0}\makebox(0,0)[lt]{\lineheight{1.25}\smash{\begin{tabular}[t]{l}a)\end{tabular}}}}%
    \put(0.37220499,0.290885){\color[rgb]{0,0,0}\makebox(0,0)[lt]{\lineheight{1.25}\smash{\begin{tabular}[t]{l}b)\end{tabular}}}}%
    \put(0.62027296,0.290885){\color[rgb]{0,0,0}\makebox(0,0)[lt]{\lineheight{1.25}\smash{\begin{tabular}[t]{l}c)\end{tabular}}}}%
    \put(0.8668862,0.290885){\color[rgb]{0,0,0}\makebox(0,0)[lt]{\lineheight{1.25}\smash{\begin{tabular}[t]{l}d)\end{tabular}}}}%
    \put(0.17039995,0.1219419){\color[rgb]{0,0,0}\makebox(0,0)[lt]{\lineheight{1.25}\smash{\begin{tabular}[t]{l}$y(t)$\end{tabular}}}}%
    \put(0.17039995,0.10283594){\color[rgb]{0,0,0}\makebox(0,0)[lt]{\lineheight{1.25}\smash{\begin{tabular}[t]{l}$\dot{y}(t)$\end{tabular}}}}%
    \put(0.17039995,0.08160709){\color[rgb]{0,0,0}\makebox(0,0)[lt]{\lineheight{1.25}\smash{\begin{tabular}[t]{l}$z_0(t)$\end{tabular}}}}%
    \put(0.17039995,0.06250114){\color[rgb]{0,0,0}\makebox(0,0)[lt]{\lineheight{1.25}\smash{\begin{tabular}[t]{l}$z_1(t)$\end{tabular}}}}%
    \put(0.39330284,0.1219419){\color[rgb]{0,0,0}\makebox(0,0)[lt]{\lineheight{1.25}\smash{\begin{tabular}[t]{l}$y(t)$\end{tabular}}}}%
    \put(0.39330284,0.10283594){\color[rgb]{0,0,0}\makebox(0,0)[lt]{\lineheight{1.25}\smash{\begin{tabular}[t]{l}$\dot{y}(t)$\end{tabular}}}}%
    \put(0.39330284,0.08160709){\color[rgb]{0,0,0}\makebox(0,0)[lt]{\lineheight{1.25}\smash{\begin{tabular}[t]{l}$z_0(t)$\end{tabular}}}}%
    \put(0.39330284,0.06250114){\color[rgb]{0,0,0}\makebox(0,0)[lt]{\lineheight{1.25}\smash{\begin{tabular}[t]{l}$z_1(t)$\end{tabular}}}}%
    \put(0.65017189,0.1219419){\color[rgb]{0,0,0}\makebox(0,0)[lt]{\lineheight{1.25}\smash{\begin{tabular}[t]{l}$y(t)$\end{tabular}}}}%
    \put(0.65017189,0.10283594){\color[rgb]{0,0,0}\makebox(0,0)[lt]{\lineheight{1.25}\smash{\begin{tabular}[t]{l}$\dot{y}(t)$\end{tabular}}}}%
    \put(0.65017189,0.08160709){\color[rgb]{0,0,0}\makebox(0,0)[lt]{\lineheight{1.25}\smash{\begin{tabular}[t]{l}$z_0(t)$\end{tabular}}}}%
    \put(0.65017189,0.06250114){\color[rgb]{0,0,0}\makebox(0,0)[lt]{\lineheight{1.25}\smash{\begin{tabular}[t]{l}$z_1(t)$\end{tabular}}}}%
    \put(0.8943036,0.1219419){\color[rgb]{0,0,0}\makebox(0,0)[lt]{\lineheight{1.25}\smash{\begin{tabular}[t]{l}$y(t)+\nu(t)$\end{tabular}}}}%
    \put(0.8943036,0.10283594){\color[rgb]{0,0,0}\makebox(0,0)[lt]{\lineheight{1.25}\smash{\begin{tabular}[t]{l}$y(t)$\end{tabular}}}}%
    \put(0.8943036,0.08160709){\color[rgb]{0,0,0}\makebox(0,0)[lt]{\lineheight{1.25}\smash{\begin{tabular}[t]{l}$w_1(t)$\end{tabular}}}}%
    \put(0.8943036,0.06250114){\color[rgb]{0,0,0}\makebox(0,0)[lt]{\lineheight{1.25}\smash{\begin{tabular}[t]{l}$z_0(t)$\end{tabular}}}}%
    \put(0.65017189,0.14317074){\color[rgb]{0,0,0}\makebox(0,0)[lt]{\lineheight{1.25}\smash{\begin{tabular}[t]{l}$y(t)+\nu(t)$\end{tabular}}}}%
    \put(0.39330288,0.14317074){\color[rgb]{0,0,0}\makebox(0,0)[lt]{\lineheight{1.25}\smash{\begin{tabular}[t]{l}$y(t)+\nu(t)$\end{tabular}}}}%
    }
  \end{picture}%
\endgroup%
    \caption{Simulation results for Example \ref{Ex:FOD}. For a) and b) we apply the algorithm~\eqref{Eq:Diff} with $n=1$ and $L=1$, to the signal $y(t)=0.75\cos(t) + 0.0025\sin(10t)+t$ for a) the noiseless case, b) with white noise $\nu(t)$ with standard deviation of $0.1$. For c) we apply the algorithm~\eqref{Eq:Diff} with $n=1$ and $L=10$, to the signal $y(t)=0.1\cos(10t) + 0.1\sin(10t)+t$ with white noise $\nu(t)$ with standard deviation of $0.1$ . For d) we apply the algorithm~\eqref{Eq:FilDiff} with  $n_f=1$, $n_d=0$, and $L=1$, to the signal $y(t)=0.75\cos(t) + 0.0025\sin(10t)+t$ with white noise $\nu(t)$ with standard deviation of $0.5$. For all simulations $\kappa(t)$ is bounded by $\kappa_{\max}=6.362$.}
    \label{fig:FOD}
\end{figure*}
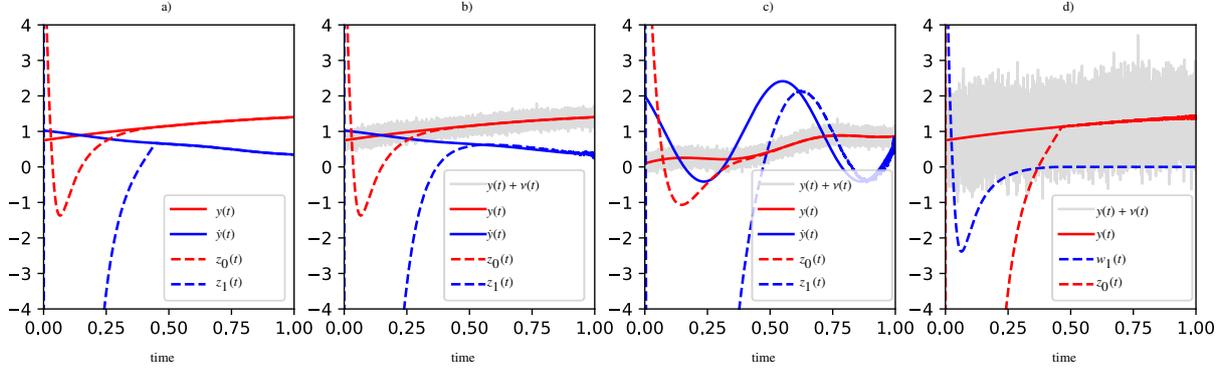
\end{example}

Recall that in~\cite[Section~5.1]{Seeber2020ExactBound}, the convergence is obtained at 0.25 while the \textit{UBST} is 1. Example~\ref{ex:seeber_comparison} illustrates how our redesign significantly reduces the overestimation and allows to tune the transient behavior. It also illustrates the case where predefined-time convergence is obtained with bounded gains.

\begin{example}
\label{ex:seeber_comparison}
In this example, we compare the performance of our differentiator with the predefined-time differentiator introduced in~\cite[Section~5.1]{Seeber2020ExactBound}. To this end, consider the same situation as in Example~\ref{Ex:FOD} using the predefined-time differentiator~\eqref{Eq:Diff} with $n=1$ and differentiating the signal $y(t)=0.75\cos(t) + 0.025\sin(10t)+t$, and notice that the differentiator in Example~\ref{Ex:FOD} can be seen as a redesigned differentiator obtained using the algorithm in~\cite[Section~5.1]{Seeber2020ExactBound} as a base differentiator.
In Figure~\ref{fig:Seeber_comparison}, we compare the performance of both algorithms for different initial conditions.
\begin{figure*}
    \centering
\def\svgwidth{16.0cm}
\begingroup%
  \makeatletter%
  \providecommand\color[2][]{%
    \errmessage{(Inkscape) Color is used for the text in Inkscape, but the package 'color.sty' is not loaded}%
    \renewcommand\color[2][]{}%
  }%
  \providecommand\transparent[1]{%
    \errmessage{(Inkscape) Transparency is used (non-zero) for the text in Inkscape, but the package 'transparent.sty' is not loaded}%
    \renewcommand\transparent[1]{}%
  }%
  \providecommand\rotatebox[2]{#2}%
  \newcommand*\fsize{\dimexpr\f@size pt\relax}%
  \newcommand*\lineheight[1]{\fontsize{\fsize}{#1\fsize}\selectfont}%
  \ifx\svgwidth\undefined%
    \setlength{\unitlength}{986.11126709bp}%
    \ifx\svgscale\undefined%
      \relax%
    \else%
      \setlength{\unitlength}{\unitlength * \real{\svgscale}}%
    \fi%
  \else%
    \setlength{\unitlength}{\svgwidth}%
  \fi%
  \global\let\svgwidth\undefined%
  \global\let\svgscale\undefined%
  \makeatother%
  \begin{picture}(1,0.21229812)%
    \lineheight{1}%
    \setlength\tabcolsep{0pt}%
    \put(0,0){\includegraphics[width=\unitlength,page=1]{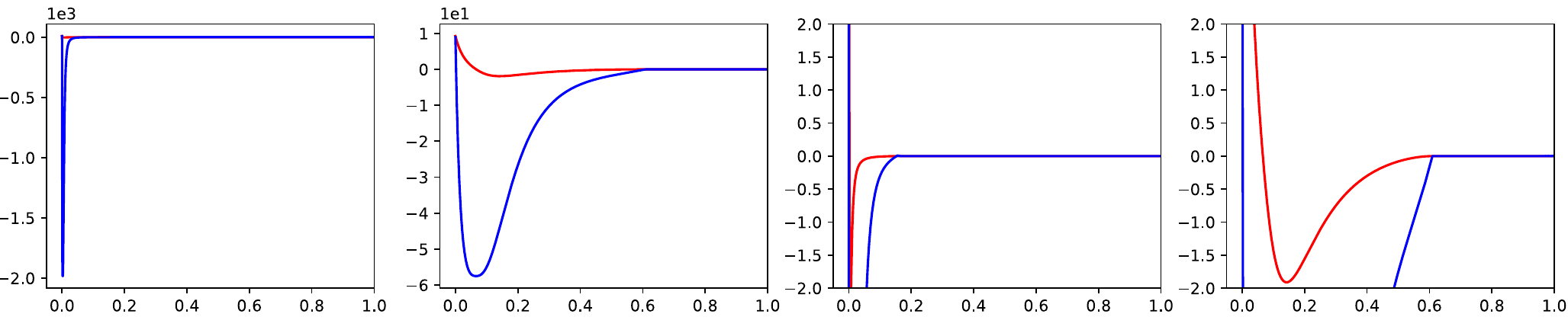}}%
    \put(0,0){\includegraphics[width=\unitlength,page=1]{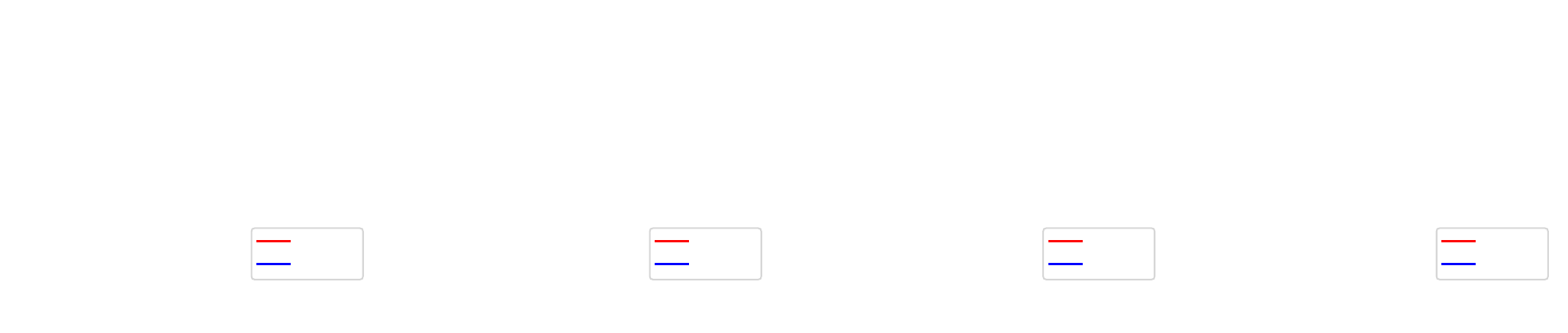}}%
    \tiny{
    \put(0.18656407,0.054109){\color[rgb]{0,0,0}\makebox(0,0)[lt]{\lineheight{1.25}\smash{\begin{tabular}[t]{l}$e_0(t)$\end{tabular}}}}%
    \put(0.18624204,0.0413909){\color[rgb]{0,0,0}\makebox(0,0)[lt]{\lineheight{1.25}\smash{\begin{tabular}[t]{l}$e_1(t)$\end{tabular}}}}%
    \put(0.44059221,0.054109){\color[rgb]{0,0,0}\makebox(0,0)[lt]{\lineheight{1.25}\smash{\begin{tabular}[t]{l}$e_0(t)$\end{tabular}}}}%
    \put(0.44027017,0.0413909){\color[rgb]{0,0,0}\makebox(0,0)[lt]{\lineheight{1.25}\smash{\begin{tabular}[t]{l}$e_1(t)$\end{tabular}}}}%
    \put(0.69157806,0.054109){\color[rgb]{0,0,0}\makebox(0,0)[lt]{\lineheight{1.25}\smash{\begin{tabular}[t]{l}$e_0(t)$\end{tabular}}}}%
    \put(0.69125602,0.0413909){\color[rgb]{0,0,0}\makebox(0,0)[lt]{\lineheight{1.25}\smash{\begin{tabular}[t]{l}$e_1(t)$\end{tabular}}}}%
    \put(0.94256394,0.054109){\color[rgb]{0,0,0}\makebox(0,0)[lt]{\lineheight{1.25}\smash{\begin{tabular}[t]{l}$e_0(t)$\end{tabular}}}}%
    \put(0.9422419,0.0413909){\color[rgb]{0,0,0}\makebox(0,0)[lt]{\lineheight{1.25}\smash{\begin{tabular}[t]{l}$e_1(t)$\end{tabular}}}}%
    \put(0.10253701,0.20459345){\color[rgb]{0,0,0}\makebox(0,0)[lt]{\lineheight{1.25}\smash{\begin{tabular}[t]{l}Seeber et al.~\cite{Seeber2020ExactBound} \end{tabular}}}}%
    \put(0.37418825,0.20370839){\color[rgb]{0,0,0}\makebox(0,0)[lt]{\lineheight{1.25}\smash{\begin{tabular}[t]{l}Ours\end{tabular}}}}%
    \put(0.60442938,0.20459345){\color[rgb]{0,0,0}\makebox(0,0)[lt]{\lineheight{1.25}\smash{\begin{tabular}[t]{l}Seeber et al.~\cite{Seeber2020ExactBound} \end{tabular}}}}%
    \put(0.87608063,0.20370839){\color[rgb]{0,0,0}\makebox(0,0)[lt]{\lineheight{1.25}\smash{\begin{tabular}[t]{l}Ours\end{tabular}}}}%
    \put(0.12388824,0.00014359){\color[rgb]{0,0,0}\makebox(0,0)[lt]{\lineheight{1.25}\smash{\begin{tabular}[t]{l}time\end{tabular}}}}%
    \put(0.37483443,0.00014359){\color[rgb]{0,0,0}\makebox(0,0)[lt]{\lineheight{1.25}\smash{\begin{tabular}[t]{l}time\end{tabular}}}}%
    \put(0.62578061,0.00014359){\color[rgb]{0,0,0}\makebox(0,0)[lt]{\lineheight{1.25}\smash{\begin{tabular}[t]{l}time\end{tabular}}}}%
    \put(0.87672681,0.00014359){\color[rgb]{0,0,0}\makebox(0,0)[lt]{\lineheight{1.25}\smash{\begin{tabular}[t]{l}time\end{tabular}}}}%
    }
  \end{picture}%
\endgroup%
    \caption{Simulation results for Example~\ref{ex:seeber_comparison}. Comparing our approach with $\alpha=5$ and $\beta=1.5(\alpha T_c/\eta)^{n+1}$ with the autonomous predefined-time differentiator in~\cite{Seeber2020ExactBound}. Here $L=1$, $y(t)=0.75\cos(t) + 0.025\sin(10t)+t$ and $z_0(0)=z_1(0)=10$. For all simulations $\kappa(t)$ is bounded by $\kappa_{\max}=29.49$.}
    \label{fig:Seeber_comparison}
\end{figure*}
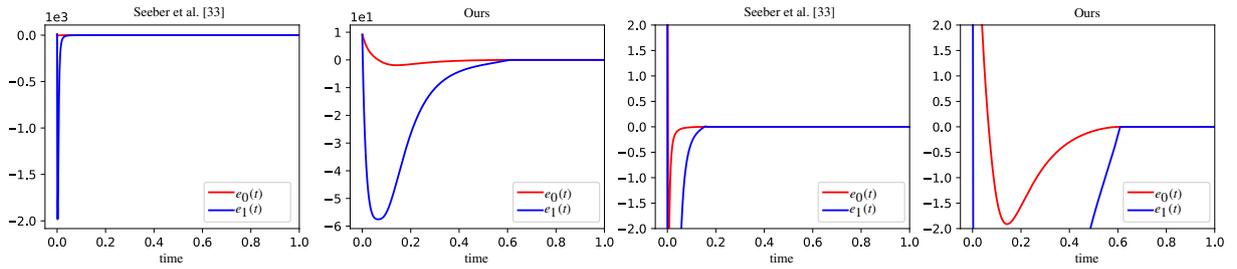
\end{example}

Since the estimate of the \textit{UBST} in~\cite{Seeber2020ExactBound} is conservative, the resulting predefined-time exact differentiator is over-engineered and converges sooner than required, producing large estimation errors.
Using the algorithm from~\cite{Seeber2020ExactBound} as a base differentiator, in the redesigned algorithm we can tune $\alpha$ and $\beta$ to reduce the slack of the \textit{UBST} and to obtain a lower maximum error than in~\cite{Seeber2020ExactBound}. As it can be observed from Figure~\ref{fig:Seeber_comparison}, both the maximum error and the slack are significantly reduced using our approach.

\section{Conclusion}
\label{Sec:Conclu}

We have introduced a methodology to design predefined-time arbitrary-order exact differentiators. Our approach uses a class of time-varying gains (\textit{TVG}), that in previous approaches diverged to infinity. In our methodology, we provided sufficient conditions to obtain a predefined-time differentiator with bounded \textit{TVG}. Such boundedness of the \textit{TVG} is shown to be necessary and sufficient for uniform (with respect to time) Lyapunov stability of the origin of the differentiator error dynamics. Furthermore, we show that the desired upper bound for the convergence can be set arbitrarily tight, contrary to existing fixed-time differentiators which yield very conservative upper bounds for the convergence time or where such upper bound is unknown.

Future work may study the optimal choice of the redesigned differentiator parameters in the presence of measurement noise~\cite{Fraguela2012DesignNoise}, which could significantly improve the performance of the differentiators. Constructing discretization methods that preserve the performance of the differentiators are also needed~\cite{Carvajal-Rubio2019OnDifferentiators}.
Overall, our results demonstrate how time-varying gains provide a very flexible methodology to design arbitrary-order differentiation algorithms with fixed-time convergence, opening the door to apply them to solve control and estimation problems with time-constraints.

\section*{Acknowledgements}
Richard Seeber gratefully acknowledges funding by the Christian Doppler Research Association, the Austrian Federal Ministry for Digital and Economic Affairs and the National Foundation for Research, Technology and Development. Marco~Tulio~Angulo gratefully acknowledges the financial support from CONACyT A1-S-13909.

\appendix
\section{On admissible  correction functions}
\begin{lemma}
\label{Lem:Aux}
The correction functions $\{\phi_i\}_{i=0}^n$ as given in Table \ref{Tab:phi} are admissible correction functions. Namely, there exist an interval $\mathcal{I}_\phi\subseteq\mathbb{R}_+$ such that, for any $\alpha\in\mathcal{I}_\phi$ the base differentiator is an asymptotic differentiator for the class $\mathcal{Y}_{\mathcal{L}(t)}^{(n+1)}$, where $\mathcal{L}(t)=L\exp{-\alpha(n+1)t}$ and (A1) holds. Moreover, $T_f$ as given in Table \ref{Tab:phi} is an \textit{UBST} for the base differentiator.
\end{lemma}
\begin{proof}
To show that Table~\ref{Tab:phi}-\textit{i)} are admissible correction functions satisfying (A1) let $\mathcal{I}_\phi=[0,r)$, any $\alpha\in\mathcal{I}_\phi$ and consider the error dynamics of the differentiator under such linear correction functions given by
$
\dot{\mathbf{e}}=\tilde{\mathbf{A}}\mathbf{e}-\mathbf{b}_{n}d(t)
$
where $|d(t)|\leq\mathcal{L}(t)$, for all $t\geq 0$, $\tilde{\mathbf{A}}=[a_{i,j}]\in\mathbb{R}^{(n+1)\times (n+1)}$ defined by $a_{i,1}=-r^{i}l_{i-1}, a_{i,i+1}=1$ and $a_{i,j}=0$ elsewhere, i.e.,
\begin{equation*}
\tilde{\mathbf{A}} = \begin{bmatrix}
- r l_0 & 1 & 0 & \ldots & 0 \\
- r^2 l_1 & 0 & 1 & \ldots & 0 \\
\vdots & \vdots & \ddots & \ddots & \vdots \\
-r^{n} l_{n-1} & 0 & 0 & \ldots & 1 \\
-r^{n+1} l_n & 0 & 0 & \ldots & 0
\end{bmatrix}
\end{equation*}
and $\mathbf{b}_{n}=[0,\dots,0,1]^T\in\mathbb{R}^{(n+1)\times 1}$. Notice that the characteristic polynomial of $\tilde{\mathbf{A}}$ is $s^{n+1}+rl_0s^{n+1}+\cdots+r^{n+1}l_n$ with eigenvalue set $\{r\lambda_0,\ldots,r\lambda_n\}$ and $\max(\real{\lambda_0},\ldots,\real{\lambda_{n}})=-(n+1)$.
Moreover, 
$
\left\|\mathbf{e}(t)\right\|=\left\|\exp{\tilde{\mathbf{A}}t}\mathbf{e}(0) + \int_0^t \exp{\tilde{\mathbf{A}}(t-s)}\mathbf{b}_{n} d(s)\mathrm{d}s\right\|.
$
Additionally, there exists a constant $c\in\mathbb{R}$ such that $\|\exp{\tilde{\mathbf{A}}t}\mathbf{e}(0)\|\leq c~\exp{-r(n+1) t}\|\mathbf{e}(0)\|$. Consequently,
$$
\left\|\mathbf{e}(t)\right\|\leq c~\exp{-r(n+1) t}\|\mathbf{e}(0)\| + \frac{c L}{(r-\alpha)(n+1)}\left(\exp{-\alpha(n+1)t}- \exp{-r (n+1) t}\right)
$$
Furthermore, let $\gamma(\mathbf{e}(0),\alpha)=c\|\mathbf{e}(0)\|+cL(r-\alpha)^{-1}(n+1)^{-1}$. Then,
$\left\|\mathbf{e}(t)\right\|\leq \gamma(\mathbf{e}(0),\alpha)\exp{-\alpha( n+1) t}$ since $\exp{-r(n+1)t}<\exp{-\alpha(n+1)t}$. Hence, the correction functions $\{\phi_i\}_{i=0}^n$ as given in Table \ref{Tab:phi}-\textit{i}) are admissible.

It follows from~\cite{Levant2003} and \cite{Seeber2020ExactBound} that with correction functions $\{\phi_i\}_{i=0}^n$ in Table~\ref{Tab:phi}-\textit{ii})-\textit{iii}) the base differentiator is an exact differentiator for a class $\mathcal{Y}_L^{(n+1)}$ with $L\neq0$, while it follows from~\cite{Menard2017} that with correction functions in Table~\ref{Tab:phi}-\textit{iv}) the base differentiator is an exact differentiator for a class $\mathcal{Y}_L^{(n+1)}$ with $L=0$. Since, $\mathcal{L}(t)=L\exp{-\alpha(n+1)t}\leq L$, for all $t\geq0$, and any $\alpha\in\mathcal{I}_\phi=[0,\infty)$, with the correction functions in Table~\ref{Tab:phi}-\textit{ii})-\textit{vi}) the base differentiator is an exact differentiator for the class $\mathcal{Y}_{\mathcal{L}(t)}^{(n+1)}$.
Moreover, pick $\gamma(\mathbf{e}(0),\alpha) = S(\mathbf{e}(0))\exp{\alpha(n+1)\mathcal{T}(\mathbf{e}(0))}$ where $S(\mathbf{e}(0))=\sup\{\|\mathbf{e}(t)\|:0\leq t\leq \mathcal{T}(\mathbf{e}(0))\}$ and $\mathcal{T}(\mathbf{e}(0))$ is the settling time of $\mathbf{e}(t)$. Thus, for $0\leq t\leq \mathcal{T}(\mathbf{e}(0))$, 
$$
\gamma(\mathbf{e}(0),\alpha)\exp{-\alpha(n+1)t}= S(\mathbf{e}(0))\exp{\alpha(n+1)(\mathcal{T}(e_0) - t)} \geq S(\mathbf{e}(0))\geq \|\mathbf{e}(t)\|
$$
And $\gamma(\mathbf{e}(0),\alpha)\exp{-\alpha(n+1)t}\geq 0=\|\mathbf{e}(t)\|$ for $t\geq \mathcal{T}(\mathbf{e}(0))$. Hence, the correction functions $\{\phi_i\}_{i=0}^n$ as given in Table \ref{Tab:phi}-\textit{ii})-\textit{iv}) are admissible. The \textit{UBST} for the convergence time of the error dynamics of the base differentiator with correction functions given in Table~\ref{Tab:phi}-\textit{iii}) and Table~\ref{Tab:phi}-\textit{iv}) is given in~\cite{Seeber2020ExactBound} and~\cite{Menard2017}, respectively.
\end{proof}

\section{Time-scale transformations}
\label{Appendix:PrelimTimScale}

The trajectories corresponding to the system solutions of~\eqref{Eq:DiffErr} are interpreted, in the sense of differential geometry~\cite{Kuhnel2015DifferentialGeometry}, as regular parametrized curves. Since we apply regular parameter transformations over the time variable, this reparametrization is referred to as time-scale transformation.

\begin{definition}\cite[Definition~2.1]{Kuhnel2015DifferentialGeometry}
\label{Def:RegularParamCurve}
A regular parametrized curve, with parameter $t$, is a $C^1(\mathcal{I})$ immersion $c: \mathcal{I}\to \mathbb{R}$, defined on a real interval $\mathcal{I} \subseteq \mathbb{R}$. This means that $\frac{\mathrm{d}c}{\mathrm{d}t}\neq 0$ holds everywhere.
\end{definition}

\begin{definition}\cite[Pg.~8]{Kuhnel2015DifferentialGeometry}
\label{Def:RegularCurve}
A regular curve is an equivalence class of regular parametrized curves, where the equivalence relation is given by regular (orientation preserving) parameter transformations $\varphi$, where $\varphi:~\mathcal{I}~\to~\mathcal{I}'$ is $C^1(\mathcal{I})$, bijective and $\frac{\mathrm{d}\varphi}{\mathrm{d}t}>0$. Therefore, if $c:\mathcal{I}\to\mathbb{R}^n$ is a regular parametrized curve and $\varphi:\mathcal{I}\to \mathcal{I}'$ is a regular parameter transformation, then $c$  and  $c\circ\varphi:\mathcal{I}'\to\mathbb{R}^n$ are considered to be equivalent.
\end{definition}

\begin{lemma}\cite{aldana2019design}
\label{Lemma:ParTrans} The bijective function $\varphi(t)=\tau=-\alpha^{-1}\ln(1-\eta t/T_c)$, defines a parameter transformation with $\varphi^{-1}(\tau)=t=\eta^{-1}T_c(1-\exp{-\alpha\tau})$ as its inverse mapping. Moreover, $\left. \frac{\mathrm{d}t}{\mathrm{d}\tau}\right|_{\tau=\varphi(t)}=\kappa(t)^{-1}$, for $t\in[0,T_c)$ and $\kappa$ given in \eqref{eq:kappa}.
\end{lemma}

\section{Proof of the main result}
\label{App:ProofMain}
Our strategy to prove the main result is as follows. First, we build an ``auxiliary system" with the selection of correction functions $\{\phi_i\}_{i=0}^n$. Then, we show that the auxiliary system and the differentiation error dynamics~\eqref{Eq:DiffErr} are related by a coordinate change and the time-scale transformation given in Lemma~\ref{Lemma:ParTrans}. Thus, the settling-time for~\eqref{Eq:DiffErr} can be obtained in the basis of the settling-time of the auxiliary system and the time-scaling.

\begin{lemma}
\label{Lemma:TimeScale}
Suppose that the conditions of Theorem~\ref{Th:MainResult} are fulfilled.
Then, the origin of the auxiliary system
\begin{align}
\label{Eq:AuxSyst}
    \frac{\mathrm{d}\chi_i}{\mathrm{d}\tau}&=-\phi_i(\chi_0)+\chi_{i+1}, \text{ for } i=0,\ldots,n-1,\\
    \frac{\mathrm{d}\chi_n}{\mathrm{d}\tau}&=-\phi_n(\chi_0)+\pi(\tau),
\end{align}
is globally asymptotically stable for $\pi(\tau)=\beta^{-1}(\alpha T_c/\eta)^{n+1}\exp{-\alpha(n+1) \tau}d(\varphi^{-1}(\tau))$ where $\varphi(t)=-\alpha^{-1}\ln(1-\eta t/T_c)$ is the parameter transformation from Lemma \ref{Lemma:ParTrans} in Appendix \ref{Appendix:PrelimTimScale}. Moreover, let $\mathbf{\chi}=[\chi_0,\ldots,\chi_n]^T$, then for every solution $\mathbf{e}(t)$ of \eqref{Eq:DiffErr} there is a solution $\mathbf{\chi}(\tau)$ of~\eqref{Eq:AuxSyst} such that
\begin{equation}
\label{eq:change_of_coordinates}
    \mathbf{e}(t) = \beta \mathbf{K}(t) \mathbf{Q} \mathbf{\chi}(\varphi(t)),
\end{equation}
holds for all $t \in [0, T_c)$ where $\mathbf{K}(t):=\textrm{diag}(1,\kappa(t),\ldots,\kappa(t)^{n})$.
That is, the curves $\beta^{-1} \mathbf{Q}^{-1} \mathbf{K}(t)^{-1} \mathbf{e}(t)$ and $\mathbf{\chi}(\tau)$ with $\mathbf{e}(0)=\beta\mathbf{K}(0)\mathbf{Q}\mathbf{\chi}(0)$ are equivalent curves under the time-scale transformation $\tau=\varphi(t)$. Thus, the redesigned differentiator's error dynamics~\eqref{Eq:Diff}, is asymptotically stable.
\end{lemma}

\begin{proof}
Denote $\tilde{\mathbf{\chi}}(t) :=  \beta^{-1}\mathbf{Q}^{-1}\mathbf{K}(t)^{-1} \mathbf{e}(t)$ and define $\mathbf{A}:=-\mathbf{Q}^{-1}(\mathbf{U}-\alpha \mathbf{D})^{n+1}\mathbf{b}_{n} \mathbf{b}_1^T$, where $\mathbf{b}_1 = [1,0,\dots,0]^T$. Due to the definition of $\kappa(t)$, we have $\kappa(t)^{-1}\dot{\kappa}(t) = \alpha\kappa(t)$ and $$\frac{\mathrm{d}}{\mathrm{d}t}\mathbf{K}(t)^{-1} = -\kappa(t)^{-1}\dot{\kappa}(t)\mathbf{D}\mathbf{K}(t)^{-1} = -\alpha \kappa(t)\mathbf{D}\mathbf{K}(t)^{-1}.$$ 
We may then write \eqref{Eq:DiffErr} as 
\begin{equation}
\label{Eq:Transfom}
\dot{e} = -\kappa(t)\mathbf{K}(t)\mathbf{Q}(\beta\Phi(\beta^{-1}e_0) - \mathbf{A}e) + \mathbf{U}e + \mathbf{b}_{n}d,
\end{equation}
in the interval $t\in[0,T_c)$.
Note that $\mathbf{Q}$, by construction, is a lower triangular matrix with ones in the main diagonal.
Hence, $e_0 = \beta\tilde{\chi}_0$ and $\mathbf{Q}^{-1}\mathbf{K}(t)^{-1}\mathbf{b}_{n} = \kappa(t)^{-n}\mathbf{b}_{n}$.
Furthermore, note that $\mathbf{Q} \mathbf{U} = (\mathbf{U} - \alpha \mathbf{D}) \mathbf{Q} + \mathbf{Q} \mathbf{A}$, and hence $\mathbf{Q}^{-1}(\mathbf{U}-\alpha\mathbf{D})\mathbf{Q} = \mathbf{U}-\mathbf{A}$.
Using these relations, as well as $\mathbf{K}(t)^{-1}\mathbf{U}\mathbf{K}(t) = \kappa(t)\mathbf{U}$, we obtain the dynamics of $\tilde{\mathbf{\chi}}$ in $[0, T_c)$ as
\begin{align}
    \begin{aligned}
    \dot{\tilde{\mathbf{\chi}}} &= \beta^{-1}\mathbf{Q}^{-1}\left(\frac{\mathrm{d}}{\mathrm{d}t}\mathbf{K}(t)^{-1}\right)e + \beta^{-1}\mathbf{Q}^{-1}\mathbf{K}(t)^{-1}\dot{e}\\ &=\kappa(t)\left(-\Phi(\tilde{\chi}_0) + \mathbf{U}\tilde{\mathbf{\chi}} + \beta^{-1}\kappa(t)^{-(n+1)}\mathbf{b}_{n}d\right).
    \end{aligned}
\end{align}
Now, consider the time-scale transformation $\tau=\varphi(t)$ to obtain an expression for $\frac{\mathrm{d}\tilde{\mathbf{\chi}}}{\mathrm{d}\tau} = \frac{\mathrm{d}}{\mathrm{d}\tau}\tilde{\mathbf{\chi}}\left(\varphi^{-1}(\tau)\right)$ which yields
$$
\frac{\mathrm{d}}{\mathrm{d}\tau}\tilde{\mathbf{\chi}} = -\Phi(\tilde{\chi}_0) + \mathbf{U}\tilde{\mathbf{\chi}} + \mathbf{b}_{n}\pi(\tau),
$$
with $\pi(\tau) = \beta^{-1}\kappa\left(\varphi^{-1}(\tau)\right)^{-(n+1)}d\left(\varphi^{-1}(\tau)\right) = \beta^{-1}(\alpha T_c/\eta)^{n+1}\exp{-\alpha(n+1)\tau}d(\varphi^{-1}(\tau))$. Comparing this to \eqref{Eq:AuxSyst} and by using $\tilde{\mathbf{\chi}}(0) = {\chi}(0)$, one can see that $\mathbf{\chi}(\tau) = \tilde{\mathbf{\chi}}\left(\varphi^{-1}(\tau)\right)$ is a solution of \eqref{Eq:AuxSyst}. Now, let
$\tilde{\mathbf{e}}(\tau):=\mathbf{e}(\varphi^{-1}(\tau))=\beta \mathbf{K}(\varphi^{-1}(\tau)) \mathbf{Q} \mathbf{\chi}(\tau)$ and notice that $\kappa(\varphi^{-1}(\tau)) =  (\alpha T_c/\eta)^{-1}\exp{\alpha \tau}$. Thus, 
\begin{align}
\mathbf{K}(\varphi^{-1}(\tau))=\textrm{diag}(1,(\alpha T_c/\eta)^{-1}\exp{\alpha\tau},\ldots,(\alpha T_c/\eta)^{-n}\exp{\alpha n\tau}).
\end{align}

Since, \eqref{Eq:AuxSyst} is asymptotically stable and by property (A1), $\mathbf{\chi}(\tau)$ satisfies
$$
\|\mathbf{\chi}(\tau)\|\leq \gamma(\mathbf{\chi}(0),\alpha) \exp{-\alpha(n+1)\tau} 
$$ for some  $\gamma(\mathbf{\chi}(0),\alpha)>0$. Moreover, we have $\|\tilde{\mathbf{e}}(\tau)\|\leq \beta(\alpha T_c/\eta)^{-n} \overline{\sigma}(\mathbf{Q})\exp{\alpha n\tau}\|\mathbf{\chi}(\tau)\|$ by the Rayleigh inequality \cite[Theorem 4.2.2]{Horn2012MatrixAnalysis} where $\overline{\sigma}(\mathbf{Q})$ and $(\alpha T_c/\eta)^{-n}\exp{\alpha n\tau}$ are the largest singular values of $\mathbf{Q}$ and $\mathbf{K}(\varphi^{-1}(\tau))$ respectively, for sufficiently large $\tau$. Then, 
$$
\lim_{\tau\to\infty}\|\tilde{\mathbf{e}}(\tau)\|\leq \lim_{\tau\to\infty}\frac{\beta \gamma(\mathbf{\chi}(0),\alpha)\overline{\sigma}(\mathbf{Q})}{(\alpha T_c/\eta)^n}\exp{-\alpha\tau}=0
$$
Thus, the redesigned differentiator's error dynamics~\eqref{Eq:DiffErr} is asymptotically stable.
\end{proof}

\begin{lemma}
\label{le:settling_time}
Suppose that functions $\{\phi_i\}_{i=0}^n$ are admissible and that the conditions of Theorem~\ref{Th:MainResult} are fulfilled. Let $\mathcal{T}(\mathbf{\chi}(0))$ be the settling time function of system~\eqref{Eq:AuxSyst}. Then, \eqref{Eq:DiffErr} is fixed-time stable with settling time function given by
\begin{equation}
\label{eq:settling_time}
T(\mathbf{e}(0))=\eta^{-1}T_c(1-\exp{-\alpha\mathcal{T}(\mathbf{\chi}(0))}),
\end{equation}
where $\mathbf{\chi}(0)=\beta^{-1}\mathbf{Q}^{-1}\mathbf{K}(0)^{-1}\mathbf{e}(0)$.
\end{lemma}
\begin{proof}
It follows from the equivalence of curves, given in Lemma~\ref{Lemma:TimeScale}, under the time-scale transformation $\tau=\varphi(t)$, that since $\mathbf{\chi}(\tau)$ reaches the origin as $\tau\to \mathcal{T}(\mathbf{\chi}(0))$, where $ \mathbf{\chi}(0)=\beta^{-1}\mathbf{Q}^{-1}\mathbf{K}(0)^{-1}\mathbf{e}(0)$, then, $\mathbf{e}(t)$ reaches the origin as $t\to\varphi^{-1}(\mathcal{T}(\mathbf{\chi}(0)))$. Thus, $T(\mathbf{e}(0))$ satisfies~\eqref{eq:settling_time}.

\end{proof}

Using these results, we are now ready to show Theorem~\ref{Th:MainResult}.

\begin{proof} (Of Theorem~\ref{Th:MainResult})
Due to Lemma \ref{le:settling_time}, the settling time function of \eqref{eq:settling_time} satisfies $T(\mathbf{e}(0))\leq T_c$. Therefore, \eqref{Eq:DiffErr}  is fixed time stable with $T_c$ as an \textit{UBST}.
\end{proof}

\subsection{Proof of the propositions}
\label{appendix_propositions}

Before showing Proposition \ref{Lemma:Uniform}, we first provide an auxiliary lemma.

\begin{lemma}
\label{lem:deltas}
Consider admissible correction functions $\{\phi\}_{i=0}^n$.
Then, for all $\delta_0 > 0$ there exists a non-zero $w \in [-\delta_0, \delta_0]$ such that no $\gamma$ exists satisfying $\Phi(w) = w \mathbf{U} \mathbf{Q}^{-1} \mathbf{b}_{0} - \gamma \mathbf{Q}^{-1} \mathbf{b}_{0}$, where $\mathbf{b}_{0}$ is defined in the notation section.
\end{lemma}
\begin{proof}
Assume the opposite, i.e., that there exists $\delta_0$ such that $\gamma(w)$ as in the lemma exists for all non-zero $w \in [-\delta_0, \delta_0]$.
We will show that this implies existence of trajectories of \eqref{Eq:AuxSyst} with $\pi(\tau) = 0$ that satisfy $\frac{\mathrm{d}}{\mathrm{d}\tau} \mathbf{b}_{n}^T \mathbf{Q} \mathbf{\chi}(\tau) = - \alpha n \mathbf{b}_{n}^T \mathbf{Q} \mathbf{\chi}(\tau)$, contradicting the convergence bound in the admissibility requirement (A1).
To see this, note that for trajectories satisfying $\abs{\chi_0(\tau)} \le \delta_0$ one has with $\mathbf{A}$ as in the proof of Lemma~\ref{Lemma:TimeScale}:
\begin{align}
    \frac{\mathrm{d}}{\mathrm{d}\tau} \mathbf{b}_{n}^T \mathbf{Q} \mathbf{\chi}&= \mathbf{b}_{n}^T \mathbf{Q} (- \Phi(\chi_0) + \mathbf{U} \mathbf{\chi}) \\
    &= \mathbf{b}_{n}^T \mathbf{Q} (\gamma \mathbf{Q}^{-1} \mathbf{b}_{0} - \chi_0 \mathbf{U} \mathbf{Q}^{-1} \mathbf{b}_{0} + \mathbf{U} \mathbf{\chi}) \\
    &= \mathbf{b}_{n}^T \mathbf{Q} \mathbf{U} \mathbf{Q}^{-1} (\mathbf{Q} \mathbf{\chi} - \chi_0 \mathbf{b}_{0})\\
    &= \mathbf{b}_{n}^T ( \mathbf{U} - \alpha \mathbf{D} + \mathbf{Q} \mathbf{A} \mathbf{Q}^{-1})(\mathbf{Q} \mathbf{\chi} - \chi_0 \mathbf{b}_{0}) \\
    &= - \alpha \mathbf{b}_{n}^T \mathbf{D} \mathbf{Q} \mathbf{\chi},
\end{align}
because $\mathbf{b}_{n}^T \mathbf{U} = 0$ and $\mathbf{A} (\mathbf{\chi} - \chi_0 \mathbf{Q}^{-1} \mathbf{b}_{0}) = 0$. Finally, use $\mathbf{b}_{n}^T\mathbf{D}=n\mathbf{b}_{n}^T$.
\end{proof}

\begin{proof}(Of Proposition \ref{Lemma:Uniform})
In the following, we  consider $n>0$ and start by showing that a bounded $\kappa(t)$, $\forall t\geq 0$ implies uniform Lyapunov stability of \eqref{Eq:DiffErr}. First, note that \eqref{Eq:DiffErr} with $\{h_i\}_{i=0}^n$ as in Theorem~\ref{Th:MainResult} is time invariant on the interval $(T_c, \infty)$.
Hence, it is sufficient to show uniform Lyapunov stability on the interval $[0, T_c)$.
Thus, let $0\leq s\leq t< T_c$ and recall the relation between $\mathbf{e}(t)$ and $\mathbf{\chi}(\varphi(t))$ in \eqref{eq:change_of_coordinates}. Note that by the Rayleigh inequality \cite[Theorem 4.2.2]{Horn2012MatrixAnalysis} we have that 
\begin{equation}
\label{eq:rayleigh}
\beta\underline{\sigma}(\mathbf{Q})\underline{\sigma}(\mathbf{K}(t))\|\mathbf{\chi}(\varphi(t))\|\leq \|\mathbf{e}(t)\|\leq\beta\overline{\sigma}(\mathbf{Q})\overline{\sigma}(\mathbf{K}(t))\|\mathbf{\chi}(\varphi(t))\|
\end{equation}
where $\underline{\sigma}(\bullet), \overline{\sigma}(\bullet)$ denote minimum and maximum singular values respectively. In addition, note that $0<\underline{\sigma}(\mathbf{K}(t))=\min\{1,\kappa(t),\dots,\kappa(t)^n\}$ and $0<\overline{\sigma}(\mathbf{K}(t))=\max\{1,\kappa(t),\dots,\kappa(t)^n\}<+\infty$ are non decreasing functions since $\kappa(t)$ is increasing for $t\in[0,T_c)$. Choose any $\epsilon>0$ and let $\epsilon_\chi = \epsilon\left(\beta\overline{\sigma}(\mathbf{Q})\overline{\sigma}(\mathbf{K}(T_c)\right)^{-1}$. Note that $\epsilon_\chi>0$ since $\overline{\sigma}(\mathbf{K}(T_c))<+\infty$. For such $\epsilon_\chi>0$, there exists $\delta_\chi>0$ such that $\|\mathbf{\chi}(\varphi(s))\|\leq \delta_\chi$ implies $\|\mathbf{\chi}(\varphi(t))\|\leq \epsilon_\chi$ for $\varphi(t)\geq \varphi(s)$ due to Lyapunov stablility and time invariance (and hence, uniform Lyapunov stability) of \eqref{Eq:BaseDiff}. Thus, let $\delta = \delta_\chi\beta\underline{\sigma}(\mathbf{Q})\underline{\sigma}(\mathbf{K}(0))$ which is independent of $s$ and satisfies $\delta\leq \delta_\chi\beta\underline{\sigma}(\mathbf{Q})\underline{\sigma}(\mathbf{K}(s))$, $\forall s\in[0,T_c)$. Hence, using the first inequality in \eqref{eq:rayleigh} it is obtained that $\|\mathbf{e}(s)\|\leq \delta\leq \delta_\chi\beta\underline{\sigma}(\mathbf{Q})\underline{\sigma}(\mathbf{K}(s))$ implies $\|\mathbf{\chi}(\varphi(s))\|\leq \delta_\chi$.
This in turn implies $\|\mathbf{\chi}(\varphi(t))\|\leq \epsilon_\chi$.
Using the second inequality in \eqref{eq:rayleigh} and the fact that $\overline{\sigma}(\mathbf{K}(t)) \leq \overline{\sigma}(\mathbf{K}(T_c))$ for all $t \in [0, T_c)$, we obtain  $\|\mathbf{e}(t)\|\leq \beta\overline{\sigma}(\mathbf{Q})\overline{\sigma}(\mathbf{K}(T_c)) \epsilon_{\chi} = \epsilon$, proving Lyapunov stability of \eqref{Eq:DiffErr} on the time interval $[0, T_c)$.

Now, we show that if $\kappa(t)$ is not bounded, then \eqref{Eq:DiffErr} is not uniformly Lyapunov stable. In particular, we will show that for any $\delta,\epsilon>0$, there exist $s,t$ with $0\leq s<t\leq T_c$ and a trajectory $e$ of \eqref{Eq:DiffErr} which satisfies both $\|\mathbf{e}(s)\|\leq \delta$ and $\|\mathbf{e}(t)\|> \epsilon$. Consider, for fixed $\delta$, arbitrary $\tau_0\ge 0$ and $\pi(\tau)=0$, any trajectory $\mathbf{\chi}(\tau)$ of \eqref{Eq:AuxSyst} with $\mathbf{\chi}(\tau_0) = w\mathbf{Q}^{-1}\mathbf{b}_{0}$ with non-zero constant $w$ as in Lemma \ref{lem:deltas} for $\delta_0 = \beta^{-1} \delta$ and $\mathbf{b}_{i}$ as given in the notation section. Now, we show that there is no real number $\gamma$ such that $\left.\frac{\mathrm{d}}{\mathrm{d}\tau}\mathbf{Q}\mathbf{\chi}(\tau)\right|_{\tau=\tau_0}= \gamma \mathbf{b}_{0}$ for this trajectory. Assume the opposite, which implies that
$\left.\frac{\mathrm{d}\mathbf{\chi}}{\mathrm{d}\tau}\right|_{\tau=\tau_0} =  -\Phi(w) +w\mathbf{U}\mathbf{Q}^{-1}\mathbf{b}_{0}=\gamma \mathbf{Q}^{-1}\mathbf{b}_{0}$. However, this is impossible due to Lemma \ref{lem:deltas} and we conclude that, for this trajectory, $\left.\frac{\mathrm{d}}{\mathrm{d}\tau}\mathbf{Q}\mathbf{\chi}(\tau)\right|_{\tau=\tau_0}\neq \gamma \mathbf{b}_{0}$ for any real $\gamma$. Therefore, there is at least one $i\in\{1,\dots,n\}$ such that $\left.\frac{\mathrm{d}}{\mathrm{d}\tau}\mathbf{b}_{i}^T\mathbf{Q}\mathbf{\chi}(\tau)\right|_{\tau=\tau_0}$ is non zero. The previous argument, in addition to the fact that \eqref{Eq:AuxSyst} is time-invariant and $\{\phi_i\}_{i=0}^n$ are continuous at $\mathbf{\chi}(\tau_0)$, implies that there exist positive constants $\tilde{\tau}$, $\tilde{\epsilon}$, which only depend on $\delta$, such that
$\mathbf{b}_{i}^T \mathbf{Q} \mathbf{\chi}(\tau_0 + \tilde{\tau})| > \tilde{\epsilon}$.
Select now $s\geq 0$ such that $\beta \kappa(s)^{i-1} \tilde{\epsilon} > \epsilon$ which is possible since $\kappa(\bullet)$ is unbounded and set $\tau_0 = \varphi(s)$. From \eqref{eq:change_of_coordinates}, one then obtains
\begin{align}
    \mathbf{e}(s) &= \beta \mathbf{K}(s) \mathbf{Q} \mathbf{\chi}(\varphi(s))= \beta \mathbf{K}(s) \mathbf{Q} \mathbf{\chi}(\tau_0)\\ &= \beta w \mathbf{K}(s) \mathbf{b}_{0} = \beta w \mathbf{b}_{0}
\end{align}
and consequently $\|\mathbf{e}(s)\|= \beta |w| \le \beta \delta_0 = \delta$.
Moreover, one has for $t = \varphi^{-1}(\varphi(s)+\tilde{\tau}) <T_c$
\begin{align}
    e_i(t) &= \mathbf{b}_{i}^T \mathbf{e}(t)= \beta \kappa(t)^{i-1} \mathbf{b}_{i}^T \mathbf{Q} \mathbf{\chi}(\varphi(t)) \\ & = \beta \kappa(t)^{i-1} \mathbf{b}_{i}^T \mathbf{Q} \mathbf{\chi}(\tau_0 + \tilde{\tau})
\end{align}
where $i\in\{1,\ldots,n\}$, and hence $\|\mathbf{e}(t)\| \ge |e_i(t)| \ge \beta \kappa(t)^{i-1} \tilde{\epsilon} \ge \beta \kappa(s)^{i-1} \tilde{\epsilon} > \epsilon$, since $\kappa(t)$ is non decreasing.
\end{proof}

\begin{proof}(Of Proposition \ref{Prop:TightBound})
Due to Lemma \ref{le:settling_time}, the settling time function of the error system \eqref{Eq:DiffErr} is given by \eqref{eq:settling_time}. Let $\mathcal{T}^* := \sup_{\mathbf{\chi}(0) \in \mathbb{R}^{n+1}}\mathcal{T}(\mathbf{\chi}(0))\leq T_f$. Hence, the least \textit{UBST} of \eqref{Eq:DiffErr} is $T^*_{\alpha}=\sup_{\mathbf{e}(0) \in \mathbb{R}^{n+1}}T(\mathbf{e}(0)) = \eta^{-1}T_c(1-\exp{-\alpha \mathcal{T}^*)}$ and thus, the slack $s_\alpha:=T_c-T_\alpha^* = T_c(1 - \eta^{-1}(1-\exp{-\alpha\mathcal{T}^*)} = \sigma(\alpha)T_c$ where $\sigma(\alpha) = 1-(1-\exp{-\alpha\mathcal{T}^*})(1-\exp{-\alpha T_f})^{-1}$. Moreover,  $\sigma(\alpha) \ge 0$ for all $\alpha$ and $\lim_{\alpha\to\infty}\sigma(\alpha) = 0$.
Hence, for every $\varepsilon$ there exists an $\alpha$ such that $\sigma(\alpha) \le T_c^{-1} \varepsilon$ and consequently $s_{\alpha} \le \varepsilon$.
\end{proof}

\begin{proof} (Proof of Proposition \ref{Prop:Bound}) Following from Lemma \ref{le:settling_time}, the settling time function of the error system \eqref{Eq:DiffErr} is given by \eqref{eq:settling_time}. For item (i), since \eqref{Eq:AuxSyst} is finite-time stable thus, we set $T_f=+\infty$ and $\eta=1$. Hence, $T(\mathbf{e}(0))< T_c$ for any finite $\|\mathbf{\chi}(0)\|$, leading to $\kappa(T(\mathbf{e}(0)))<+\infty$. For item (ii), since $\mathcal{T}(\mathbf{\chi}(0))\leq \hat{T}$ for some $\hat{T}<\infty$ then, $T(\mathbf{e}(0)) \leq T_{\max}:= \eta^{-1}T_c(1-\exp{-\alpha \hat{T}})<T_c$ regardless of $\mathbf{e}(0)$. Henceforth, $\sup_{\mathbf{e}(0)\in\mathbb{R}^{n+1}}\kappa(T(\mathbf{e}(0)))<\infty$. Finally, if we have $\mathcal{T}(\mathbf{\chi}(0))\leq T_f<\infty$ then, $T(\mathbf{e}(0)) \leq
\eta^{-1}T_c(1-\exp{-\alpha T_f}) = T_c$ independently of $\mathbf{\chi}(0)$. Hence, it can be obtained that $\kappa_{\max} \equiv \kappa(T_c)$ and  $\kappa(t)\leq \kappa_{\max}< \infty$ for $t\in [0,T_c)$.
\end{proof}

\end{document}